\documentclass[11pt]{amsart}

\usepackage{amssymb,amsmath,amsthm}
\usepackage{graphicx}
\usepackage{color}
\usepackage{thmtools}
\usepackage{thm-restate}
\usepackage[shortlabels]{enumitem}
\usepackage{mathrsfs}
\usepackage[utf8]{inputenc}
\usepackage[T1]{fontenc}
\usepackage{dsfont}
\usepackage{soul}
\usepackage[margin=1in]{geometry}

\newtheorem{theorem}{Theorem}[section]
\newtheorem{lemma}[theorem]{Lemma}
\newtheorem{proposition}[theorem]{Proposition}

\newtheorem{corollary}[theorem]{Corollary}

\theoremstyle{definition}
\newtheorem{definition}[theorem]{Definition}
\newtheorem{example}[theorem]{Example}

\theoremstyle{remark}
\newtheorem{remark}[theorem]{Remark}

\numberwithin{equation}{section}

\newcommand{\frakc}{\mathfrak{c}}

\newcommand{\eps}{\varepsilon}

\newcommand{\N}{\mathbb{N}}
\newcommand{\R}{\mathbb{R}}
\newcommand{\Q}{\mathbb{Q}}

\newcommand{\aA}{\mathcal{A}}

\newcommand{\dD}{\mathcal{D}}

\newcommand{\fF}{\mathcal{F}}

\newcommand{\kK}{\mathcal{K}}

\newcommand{\zZ}{\mathcal{Z}}

\DeclareMathOperator{\der}{\!\mathrm{d}\!}

 \DeclareMathOperator{\intt}{int}

\newcommand{\ol}{\overline}

\newcommand{\rstr}{\restriction}

\DeclareMathOperator{\sgn}{sgn}
\newcommand{\sm}{\setminus}
\newcommand{\sub}{\subseteq}
\DeclareMathOperator{\supp}{supp}

\newcommand{\seq}[2]{\big\langle#1\colon\ #2\big\rangle}
\newcommand{\seqn}[1]{\big\langle#1\colon\ n\io\big\rangle}

\newcommand{\seqk}[1]{\big\langle#1\colon\ k\io\big\rangle}

\newcommand{\seqi}[1]{\big\langle#1\colon\ i\io\big\rangle}

\newcommand{\ctblsub}[1]{\left[#1\right]^\omega}

\newcommand{\iA}{\in\aA}

\newcommand{\io}{\in\omega}

\newcommand{\wo}{{\wp(\omega)}}
\newcommand{\bo}{{\beta\omega}}

\newcommand{\cso}{\ctblsub{\omega}}

\newcommand{\noproof}{\hfill$\Box$}

\begin{document}

\title[On sequences of finitely supported measures related to the J.--N. theorem]{On sequences of finitely supported measures related to the Josefson--Nissenzweig theorem}
\author[W. Marciszewski]{Witold Marciszewski}
\address{Institute of Mathematics and Computer Science, University of Warsaw, Warsaw, Poland.}
\email{wmarcisz@mimuw.edu.pl}
\author[D.\ Sobota]{Damian Sobota}
\address{Kurt G\"odel Research Center, Institut f\"ur Mathematik, Universit\"at Wien, Wien, Austria.}
\email{ein.damian.sobota@gmail.com}
\urladdr{www.logic.univie.ac.at/~{}dsobota}
\author[L. Zdomskyy]{Lyubomyr Zdomskyy}
\address{Institut für Diskrete Mathematik und Geometrie, Technische Universität Wien, Wien, Austria.}
\email{lzdomsky@gmail.com}
\urladdr{dmg.tuwien.ac.at/zdomskyy}
\thanks{The research of the first  named author is supported by the NCN (National Science Centre, Poland) research grant no. 2020/37/B/ST1/02613.. The second and third named authors have been supported by the Austrian Science Fund FWF, Grants I 2374-N35, I 3709-N35, M 2500-N35, I 4570-N35.}

\begin{abstract}
Given a Tychonoff space $X$, we call a sequence $\langle\mu_n\colon n\in\omega\rangle$ of signed Borel measures on $X$ \textit{a finitely supported Josefson--Nissenzweig sequence} (in short \textit{a JN-sequence}) if: 1) for every $n\in\omega$ the measure $\mu_n$ is a finite combination of one-point measures and $\|\mu_n\|=1$, and 2) $\int_Xf\der\mu_n\to0$ for every continuous function $f\in C(X)$. Our main result asserts that if a Tychonoff space $X$ admits a JN-sequence, then there exists a JN-sequence $\langle\mu_n\colon n\in\omega\rangle$ such that: i) $\supp(\mu_n)\cap\supp(\mu_k)=\emptyset$ for every $n\neq k\in\omega$, and ii) the union $\bigcup_{n\in\omega}\supp(\mu_n)$ is a discrete subset of $X$. We also prove that 
if a Tychonoff space $X$ carries a JN-sequence, then either there is a JN-sequence $\langle\mu_n\colon n\io\rangle$ on $X$ such that $|\supp(\mu_n)|=2$ for every $n\in\omega$, or for every JN-sequence $\langle\mu_n\colon n\io\rangle$ on $X$ we have $\lim_{n\to\infty}|\supp(\mu_n)|=\infty$.

\end{abstract}


\maketitle

\section{Introduction}

The classical Josefson--Nissenzweig theorem states that every infinite-dimensional Banach space $X$ admits a sequence $\seqn{x_n^*}$ of continuous functionals such that $\big\|x_n^*\big\|=1$ for every $n\io$ and $x_n^*(x)\to0$ for every $x\in X$ (see \cite{Jos75} and \cite{Nis75}; cf. also \cite{HJ77}, \cite{BD84}, \cite{Beh95}). The theorem has found numerous applications in Banach space theory, see e.g. \cite{Khu78}, \cite{Cem84}, \cite{Fre84}, \cite{Bat92}, \cite{BF93}. Its validity was also studied in more general settings, e.g. for Fr\'echet spaces---see \cite{Bon91}, \cite{BLV}, \cite{LS93}.

In the class of $C_p(X)$-spaces the theorem was first studied by Banakh, K\k{a}kol, and \'{S}liwa in \cite{BKS19}, primarily in the context of the Separable Quotient Problem for $C_p(X)$-spaces. It was proved there that, given a Tychonoff space $X$, the space $C_p(X)$ contains a complemented copy of the space $(c_0)_p=\{x\in\R^\omega\colon\ x(n)\to0\}$, endowed with the pointwise topology inherited from $\R^\omega$, if and only if there exists a sequence $\seqn{\mu_n}$ of finitely supported signed Borel measures on $X$ such that $\big\|\mu_n\big\|=1$ for every $n\io$ and $\int_Xf\der\mu_n\to0$ for every continuous function $f\in C(X)$. Here, by a \textit{finitely supported} measure we mean a measure which is a finite linear combination of one-point measures, see Section \ref{section:notation} for details. Since there is a natural one-to-one linear correspondence between continuous functionals on the topological vector space $C_p(X)$ and finitely supported signed measures on $X$, the latter result can be considered as a characterization of those $C_p(X)$-spaces for which the Josefson--Nissenzweig theorem holds. To simplify the further discussion, let us introduce the following definition.

\begin{definition}\label{def:fsjnseq}
Let $X$ be a Tychonoff space.
	 A sequence $\seqn{\mu_n}$ of finitely supported signed Borel measures on $X$ is \textit{a (finitely supported) Josefson--Nissenzweig sequence} (in short, \textit{a JN-sequence}) if $\|\mu_n\|=1$ for every $n\io$ and $\int_Xf\der\mu_n\to0$ for every continuous function $f\in C(X)$.\footnote{Let us note that in the paper \cite{KSZgroth} we used the more accurate abbreviation \textit{fsJN-sequence}. The reason behind that was that in the latter paper we also studied sequences of measures with infinite supports. Since in the current paper we only focus on sequences of finitely supported measures, we decided to stick to the simpler abbreviation \textit{JN-sequence}; this approach also agrees with papers \cite{KMSZ} and \cite{MSnf}.}
\end{definition}
Thus, the main result of \cite{BKS19} asserts that a Tychonoff space admits a JN-sequence if and only if $C_p(X)$ contains a complemented copy of the space $(c_0)_p$. To provide examples, it was observed in \cite{BKS19} that, e.g., every Tychonoff space containing a non-trivial convergent sequence carries a JN-sequence but the \v{C}ech--Stone compactification $\bo$ of the set $\omega$ of natural numbers does not. For further examples and counterexamples, see \cite{KSZgroth}, \cite{KSZprod}, \cite{KMSZ}, \cite{BKScpxe}, \cite{MSnf}, where various criteria for spaces to admit JN-sequences were given.

In this paper we are interested in finitely supported Josefson--Nissenzweig sequences on Tychonoff spaces \textit{per se}, that is, we are curious to what extent we can manipulate them, change them, and, ultimately and most importantly, simplify them. Such investigations may find applications (and in fact have already found), e.g., in the studies of Grothendieck $C(K)$-spaces (see \cite{KSZgroth} and \cite{MSnf}) or in the context of the aforementioned Separable Quotient Problem for $C_p$-spaces (\cite{BKScpxe}).

Our main result reads as follows:

\begin{theorem}\label{theorem:main}
If a Tychonoff space $X$ carries a JN-sequence $\seqn{\mu_n}$, then there exists a JN-sequence $\seqn{\nu_n}$ on $X$ such that:
\begin{enumerate}\setcounter{enumi}{-1}
	\item $\bigcup_{n\io}\supp\big(\nu_n\big)\sub\bigcup_{n\io}\supp\big(\mu_n\big)$,
	\item $\supp\big(\nu_n\big)\cap\supp\big(\nu_k\big)=\emptyset$ for every $n\neq k\io$,
	\item $\bigcup_{n\io}\supp\big(\nu_n\big)$ is a discrete subset of $X$.
\end{enumerate}
\end{theorem}

This provides a great simplification as JN-sequences may initially be very complicated---see Proposition \ref{prop:square_fsjnseqs} and its proof for examples. As a corollary, we get that, at least in the case of compact Hausdorff spaces, the study of spaces carrying JN-sequences may be confined only to compactifications of $\omega$.


\begin{corollary}\label{cor:main}
Let $K$ be an infinite compact Hausdorff space (or, more generally, an infinite normal space). Then, $K$ carries a JN-sequence if and only if there exists a countable discrete subset $D$ of $K$ such that the closure $\ol{D}^K$ carries a disjointly supported JN-sequence $\seqn{\mu_n}$ with $\supp\big(\mu_n\big)\sub D$ for every $n\io$.
\end{corollary}

It is natural to ask whether a Tychonoff space admitting a JN-sequence carries also one with supports having cardinality bounded by some constant $M\io$. The answer is negative---in Section \ref{section:sizes_two_examples} we provide an example of a Boolean algebra $\dD$ such that its Stone space $St(\dD)$ carries a JN-sequence and every JN-sequence $\seqn{\mu_n}$ on $St(\dD)$ satisfies the equality $\lim_{n\to\infty}\big|\supp\big(\mu_n\big)\big|=\infty$. It appears however that if sizes of supports of a given JN-sequence on a Tychonoff space are all bounded by some $M$, then we can actually find a JN-sequence with 2-element supports.

\begin{theorem}\label{theorem:main2}
Let $X$ be a Tychonoff space $X$. If there exist a JN-sequence $\seqn{\mu_n}$ on $X$ and a number $M\io$ such that $\big|\supp\big(\mu_n\big)\big|\le M$ for every $n\io$, then $X$ admits a JN-sequence $\seqn{\nu_n}$ such that $\big|\supp\big(\nu_n\big)\big|=2$ for every $n\io$. Consequently, there exist two disjoint sequences $\seqn{x_n}$ and $\seqn{y_n}$ of distinct points in $X$ such that
\[\lim_{n\to\infty}\big(f\big(x_n\big)-f\big(y_n\big)\big)=0\]
for every $f\in C(X)$.
\end{theorem}

\begin{corollary}\label{cor:main2}
If a Tychonoff space $X$ carries a JN-sequence, then either there is a JN-sequence $\seqn{\mu_n}$ on $X$ such that $\big|\supp\big(\mu_n\big)\big|=2$ for every $n\io$, or each JN-sequence $\seqn{\mu_n}$ on $X$ satisfies the equality $\lim_{n\to\infty}\big|\supp\big(\mu_n\big)\big|=\infty$.
\end{corollary}

\medskip 

The paper is organised as follows. In the next section we briefly recall standard definitions and notions. In Section \ref{section:jn_sequences} we provide basic topological properties of JN-sequences on Tychonoff spaces. Section \ref{section:fsjn_disjoint_supps} is devoted to prove that if a Tychonoff space carries a JN-sequence, then it carries one with disjoint supports (Theorem \ref{theorem:disjoint_supps}). In Section \ref{section:discrete} we go further and prove that we can even find a JN-sequence with discrete union of supports (Theorem \ref{theorem:discrete}). In the last section, Section \ref{section:sizes_of_supports}, we study possible cardinalities of supports of JN-sequences.

\section{Preliminaries\label{section:notation}}

%

If $X$ is a set and $A$ its subset, then $A^c=X\sm A$, and $\chi_A$ denotes the characteristic function of $A$ in $X$. We also set $1_X=\chi_X$, that is, $1_X$ is the constant one function on $X$. The cardinality of a set $X$ is denoted by $|X|$. $\omega$ denotes the first infinite cardinal number and $\frakc$ denotes the continuum, i.e., the size of the real line $\R$.

\medskip

Throughout the paper, we assume that all topological spaces we consider are \textbf{Tychonoff}, so in particular every compact space we deal with is normal. 
If $X$ is a (Tychonoff) space and $A$ its subspace, then $\ol{A}^X$ denotes the closure of $A$ in $X$. 
$\beta X$ denotes the \v{C}ech--Stone compactification of $X$. 
We also usually identify $\omega$ with the discrete space $\N$ of natural numbers.

If $X$ is a space, then by $C(X)$ we denote the space of real-valued continuous functions on $X$. For $a<b\in\mathbb{R}$ we also set $C(X,[a,b])=\big\{f\in C(X)\colon \forall x\in X,\ a\le f(x)\le b\}$. For every $f\in C(X)$ we set $\|f\|_\infty=\sup\big\{|f(x)|\colon x\in X\big\}$. By $C_p(X)$ we denote the space $C(X)$ endowed with the pointwise topology (i.e., the topology inherited from the product space $\R^X$). 

\medskip

Concerning measures on Tychonoff spaces, we will only deal with finite Borel ones. Let $X$ be a space. For any (finite Borel) measure $\mu$ on $X$ and a $\mu$-integrable real-valued function $f$ on $X$, we briefly set $\mu(f)=\int_X f\der\mu$. For every $x\in X$ by $\delta_x$ we mean \textit{the point measure} (or \textit{the Dirac measure}) concentrated at $x$ and defined as $\delta_x(A)=\chi_A(x)$. A measure $\mu$ on $X$ is \textit{finitely supported} if it can be written as a finite linear combination of point measures, i.e., there exist finite sequences $x_0,\ldots,x_n$ of distinct points in $X$ and $\alpha_0,\ldots,\alpha_n$ of non-zero real numbers such that:
\[\mu=\sum_{i=0}^n\alpha_i\cdot\delta_{x_i}.\]
For such measure $\mu$, its \textit{support} $\supp(\mu)$ is the set $\big\{x_0,\ldots,x_n\big\}$, and \textit{the variation} of $\mu$ is given by the formula
\[|\mu|=\sum_{x\in\supp(\mu)}\big|\alpha_x\big|\cdot\delta_x,\]
hence the norm $\|\mu\|$ of $\mu$ is equal to $\sum_{x\in\supp(\mu)}\big|\alpha_x\big|$. For any real-valued function $f$ on $X$ we have:
\[\mu(f)=\int_X f\der\mu = \sum_{x\in\supp(\mu)}\alpha_xf(x).\]

\medskip

The following definition is crucial for our paper.
\begin{definition}\label{def:seq_measures}
If $\mu_n$ is a finitely supported measure on a space $X$, for every $n\in \omega$, then we say that the sequence $\seqn{\mu_n}$ is \textit{finitely supported}. A finitely supported sequence $\seqn{\mu_n}$ is
\begin{enumerate}
\item \textit{weak* convergent} to a finitely supported measure $\mu$ on $X$ if\newline $\lim_{n\to\infty}\mu_n(f)=\mu(f)$ for every $f\in C(X)$, 
\item \textit{weak* null} if it is weak* convergent to the zero measure $0$ on $X$.
\end{enumerate}
\end{definition}

\section{JN-sequences of measures\label{section:jn_sequences}}

This section is devoted to the study of basic analytic and topological properties of JN-sequences. 
The first lemma shows that measures in a JN-sequence have eventually similar absolute values on their negative and positive parts, equal to $\approx1/2$. It follows immediately from the definition of a JN-sequence applied for the constant function $1_X$ on $X$.

\begin{lemma}\label{lemma:jnseq_pos_neg}
Let $\seqn{\mu_n}$ be a JN-sequence on a space $X$. For every $n\io$ let $P_n=\big\{x\in\supp\big(\mu_n\big)\colon\ \mu_n(\{x\})>0\big\}$ and $N_n=\supp\big(\mu_n\big)\sm P_n$. Then,
\[\lim_{n\to\infty}\big\|\mu_n\rstr P_n\big\|=\lim_{n\to\infty}\big\|\mu_n\rstr N_n\big\|=1/2.\]\noproof
\end{lemma}

For a given finitely supported sequence $\seqn{\mu_n}$ of measures on a space $X$, let us put:
\[S\big(\seqn{\mu_n}\big)=\bigcup_{n\io}\supp\big(\mu_n\big),\]
\[LS\big(\seqn{\mu_n}\big)=\Big\{x\in X\colon\ \limsup_{n\to\infty}\big|\mu_n(\{x\})\big|>0\Big\},\]
\[LI\big(\seqn{\mu_n}\big)=\Big\{x\in X\colon\ \liminf_{n\to\infty}\big|\mu_n(\{x\})\big|>0\Big\},\]
and
\[L\big(\seqn{\mu_n}\big)=\Big\{x\in X\colon\ \lim_{n\to\infty}\mu_n(\{x\})\text{ exists and is not }0\Big\}.\]
We will usually write shorter $S\big(\mu_n\big)$, $LS\big(\mu_n\big)$, $LI\big(\mu_n\big)$, and $L\big(\mu_n\big)$ instead of $S\big(\seqn{\mu_n}\big)$, $LS\big(\seqn{\mu_n}\big)$, $LI\big(\seqn{\mu_n}\big)$, and $L\big(\seqn{\mu_n}\big)$, or even simply $S$, $LS$, $LI$, and $L$ if the sequence $\seqn{\mu_n}$ is clear from the context. Of course, always $L\sub LI\sub LS\sub S$, but the reverse inclusions may not hold (cf. Proposition \ref{prop:square_fsjnseqs}).

\begin{lemma}\label{lemma:s_infinite}
If $\seqn{\mu_n}$ is a JN-sequence on a space $X$, then $S$ is infinite.
\end{lemma}
\begin{proof}
If $S$ is finite, then there exists $x_0\in S$ and $\eps>0$ such that $\limsup_{n\to\infty}\big|\mu_n\big(\big\{x_0\big\}\big)\big|>\eps$ (if not, then there is $N\io$ such that $\big|\mu_n(\{x\})\big|<1/|S|$ for every $x\in S$ and $n>N$, which implies that $\big\|\mu_n\big\|<1$ for every $n>N$). Let $f\in C(X)$ be such that $f(x_0)=1$ and $f(x)=0$ for every $x\in S\sm\big\{x_0\big\}$. It follows that $\limsup_{n\to\infty}\big|\mu_n(f)\big|>\eps$, which is a contradiction.
\end{proof}

Note that despite the fact that the set $S$ is a countable subset of $X$ its topology may be very hard to study---see e.g. \cite{Lev77}, where it was proved that there exist $2^\frakc$ many non-homeomorphic countable regular (hence normal) spaces without points of countable character. Also, in \cite{MSnf} we provided a description of $2^\frakc$ many non-homeomorphic countable regular spaces which admit JN-sequences and have only one limit point.

\begin{remark}\label{remark:jn_pointwise_limits}
Let $\seqn{\mu_n}$ be a JN-sequence on a given space $X$. Then, since $S$ is countable, by induction we can find a subsequence $\seqk{\mu_{n_k}}$ such that $\lim_{k\to\infty}\big|\mu_{n_k}(\{x\})\big|$ exists for every $x\in X$. Denote each such limit by $\mu(\{x\})$. Then, $\mu=\sum_{x\in S}\alpha_x\cdot\delta_x$ for some $\alpha_x\in\R$, $x\in S$, and $\|\mu\|=\sum_{x\in S}\big|\alpha_x\big|\le1$. To see the latter, note that for every finite $F\sub S$ we have:
\[\|\mu\rstr F\|=\sum_{x\in F}|\mu(\{x\})|=\lim_{k\to\infty}\sum_{x\in F}\big|\mu_{n_k}(\{x\})\big|=\lim_{k\to\infty}\big\|\mu_{n_k}\rstr F\big\|\le 1.\]
\end{remark}

\begin{definition}
A sequence $\seqn{\mu_n}$ of finitely supported measures on a space $X$ is \textit{pointwise convergent} if the limit $\lim_{n\to\infty}\mu_n\big(\{x\}\big)$ exists for every $x\in X$.
\end{definition}

Note that the definition is equivalent to say that $\lim_{n\to\infty}\mu_n\big(\{x\}\big)=0$ for every $x\in X\sm L$. It follows that $L\big(\mu_n\big)=LI\big(\mu_n\big)=LS\big(\mu_n\big)\sub S\big(\mu_n\big)$ if $\seqn{\mu_n}$ is pointwise convergent. By the previous remark, every JN-sequence $\seqn{\mu_n}$ on a space $X$ contains a pointwise convergent JN-(sub)sequence $\seqk{\mu_{n_k}}$. Of course, every subsequence of a pointwise convergent sequence of measures is also pointwise convergent.

The proof of the following lemma is left to the reader.

\begin{lemma}\label{lemma:fsjn_subsequences_sets}
For every finitely supported sequence $\seqn{\mu_n}$ of measures on a space $X$ and its subsequence $\seqk{\mu_{n_k}}$ it holds:
\begin{enumerate}[(i)]
    \item $S\big(\seqk{\mu_{n_k}}\big)\sub S\big(\seqn{\mu_n}\big)$;
    \item $LS\big(\seqk{\mu_{n_k}}\big)\sub LS\big(\seqn{\mu_n}\big)$;
    \item $LI\big(\seqn{\mu_n}\big)\sub LI\big(\seqk{\mu_{n_k}}\big)$;
    \item $L\big(\seqn{\mu_n}\big)\sub L\big(\seqk{\mu_{n_k}}\big)$.
\end{enumerate}
If $\seqn{\mu_n}$ is pointwise convergent, then
\[L\big(\seqn{\mu_n}\big)=L\big(\seqk{\mu_{n_k}}\big)=LS\big(\seqk{\mu_{n_k}}\big)=LS\big(\seqn{\mu_n}\big).\]\noproof
\end{lemma}

The following proposition asserts that the unit square $[0,1]^2$ admits JN-sequences satisfying various proper inclusions between sets $L$, $LI$, $LS$, and $S$, as well as they have other quantitative properties. It also shows that even in the case of a metric space a JN-sequence may be quite intricate.

\begin{proposition}\label{prop:square_fsjnseqs}
Let $\alpha\in[0,1)$. The unit square $[0,1]^2$ admits JN-sequences $\seqn{\mu_n^1}$, $\seqn{\mu_n^2}$, $\seqn{\mu_n^3}$ and $\seqn{\mu_n^4}$ such that:
\begin{enumerate}
    \item 
$\emptyset\neq L\big(\mu_n^1\big)\subsetneq LI\big(\mu_n^1\big)\subsetneq LS\big(\mu_n^1\big)\subsetneq S\big(\mu_n^1\big)$;
    \medskip
    \item 
    \begin{enumerate}[(i)]
        \item $LS\big(\mu_n^2\big)=\big([0,1]\cap\Q\big)\times\{0\}$, so $LS\big(\mu_n^2\big)$ is dense-in-itself;
        \item $\emptyset=L\big(\mu_n^2\big)=LI\big(\mu_n^2\big)\subsetneq LS\big(\mu_n^2\big)\subsetneq S\big(\mu_n^2\big)$;
        \item $\mu_n^2(\{x\})\in\{0,1/2\}$ for every $x\in LS\big(\mu_n^2\big)$ and $n\io$;
        \item for every $x\in LS\big(\mu_n^2\big)$ we have $\limsup_{n\to\infty}\mu_n^2(\{x\})=1/2$, so for every finite $F\sub LS\big(\mu_n^2\big)$ it holds:
        \[\sum_{x\in F}\limsup_{n\to\infty}\mu_n^2(\{x\})=|F|/2,\]
        and hence:
        \[\sum_{x\in LS(\mu_n^2)}\limsup_{n\to\infty}\mu_n^2(\{x\})=\infty;\]
    \end{enumerate}
    \medskip
    \item 
    \begin{enumerate}[(i)]
        \item $L\big(\mu_n^3\big)=\big([0,1]\cap\Q\big)\times\{0\}$, so $L\big(\mu_n^3\big)$ is dense-in-itself;
        \item $\emptyset\neq L\big(\mu_n^3\big)=LI\big(\mu_n^3\big)=LS\big(\mu_n^3\big)\subsetneq S\big(\mu_n^3\big)$;
        \item 
        \[\sum_{x\in L(\mu_n^3)}\lim_{n\to\infty}\mu_n^3(\{x\})=(1-\alpha)/2\le1/2\]
        and
        \[\lim_{n\to\infty}\big\|\mu_n\rstr L\big\|=(1-\alpha)/2\le1/2;\]
    \end{enumerate}
    \medskip
    \item 
    \begin{enumerate}[(i)]
        \item $L\big(\mu_n^4\big)=\big\{k/2^{n+1}\colon\ k,n\io,\ 0\le k<2^{n+1}\big\}\times\{0\}$;
        \item $\emptyset\neq L\big(\mu_n^4\big)=LI\big(\mu_n^4\big)=LS\big(\mu_n^4\big)=S\big(\mu_n^4\big)$;
        \item $\big\|\mu_n^4\rstr L\big\|=1$ for every $n\io$.
    \end{enumerate}
\end{enumerate}
\end{proposition}
\begin{proof}
Put $K=[0,1]^2$ and fix an enumeration (without repetitions) $\big\{q_n\colon\ n\io\big\}$ of $[0,1]\cap\Q$.

\medskip

(1) If $n\io$ is even, then let $\mu_n^1$ be defined as follows:
    \[\mu_n^1={\textstyle\frac{1}{4}}\big(\delta_{(0,0)}-\delta_{(0,1/(n+1))}\big)+{\textstyle\frac{1}{4}}\big(\delta_{(1/2,0)}-\delta_{(1/2,1/(n+1))}\big),\]
    and if $n$ is odd, then define $\mu_n^1$ as follows:
    \[\mu_n^1={\textstyle\frac{1}{4}}\big(\delta_{(0,0)}-\delta_{(0,1/(n+1))}\big)+{\textstyle\frac{1}{8}}\big(\delta_{(1/2,0)}-\delta_{(1/2,1/(n+1))}\big)+{\textstyle\frac{1}{8}}\big(\delta_{(1,0)}-\delta_{(1,1/(n+1))}\big).\]
    It is immediate that $\seqn{\mu_n^1}$ is a JN-sequence on $K$ and:
    \[L\big(\mu_n^1\big)=\big\{(0,0)\big\},\]
    \[LI\big(\mu_n^1\big)=\big\{(0,0),\ (1/2,0)\big\},\]
    \[LS\big(\mu_n^1\big)=\big\{(0,0),\ (1/2,0),\ (1,0)\big\},\]
    \[S\big(\mu_n^1\big)=\big\{(0,0),\ (1/2,0),\ (1,0)\big\}\cup\big\{(x,1/(n+1)\colon\ x\in\{0,1/2,1\},\ n\io\big\},\]
    which yields (1).

    \medskip

(2) Let $\big\{P_n\colon\ n\io\big\}$ be a partition of $\omega$ into infinite sets. For every $n\io$ and $k\in P_n$ write:
    \[\mu_k^2={\textstyle\frac{1}{2}}\big(\delta_{(q_n,0)}-\delta_{(q_n,1/(k+1))}\big).\]
    Then, for each $k\io$ we have $\big\|\mu_k^2\big\|=1$ and it is immediate that for every $n\io$ the sequence $\seq{\mu_k^2}{k\in P_n}$ is weak* null. Since every $f\in C(K)$ is uniformly continuous, one can easily verify that the whole sequence $\seqk{\mu_k^2}$ is also weak* null. 

    That the conditions (i)--(iv) are satisfied follows directly from the definition of the sequence $\seqn{\mu_n^2}$.

    \medskip

(3) 
    For every $n\io$ define the measure $\mu_n^3$ as follows:
    \[\mu_n^3=(1-\alpha)\cdot\sum_{k=0}^n\big(\delta_{(q_k,0)}-\delta_{(q_k,1/(n+1))}\big)/2^{k+2}+\Big(\frac{\alpha}{2}+\frac{1-\alpha}{2^{n+2}}\Big)\cdot\big(\delta_{(0,1-1/(n+1))}-\delta_{(0,1-1/(n+2))}\big).\]
    It follows that $\big\|\mu_n^3\big\|=1$. 
    That $\seqn{\mu_n^3}$ is weak* null follows again from the fact that every $f\in C(K)$ is uniformly continuous.

    For every $k\io$ and $n\ge k$ we have:
    \[\tag{$*$}\mu_n^3\big(\big\{(q_k,0)\big\}\big)=(1-\alpha)/2^{k+2},\]
    so $\big(q_k,0\big)\in L\big(\mu_n^3\big)$. If $x\in K$ is of the form $\big(q_k,1/(n+1)\big)$ or $\big(0,1-1/n\big)$ for some $k,n\io$, then $\mu_l^3(\{x\})=0$ for every $l>n+2$, so $x\not\in L\big(\mu_n^3\big)$. Thus, (i) is satisfied. (ii) follows immediately from (i) and the definition of $\seqn{\mu_n^3}$. (iii) follows from ($*$).

    \medskip

(4) Let $n\io$. Put $P_n=\big\{0,\ldots,2^n-1\big\}$ and for each $k\in P_n$ write $e_k^n=(2k)/2^{n+1}$ and $o_k^n=(2k+1)/2^{n+1}$. Note that $e_0^n=0$. Put: $E_n=\big\{e_k^n\colon k\in P_n\big\}$, $O_n=\big\{o_k^n\colon k\in P_n\big\}$ and $S_n=E_n\cup O_n$. The set $S_n\times\{0\}$ will be the support of the measure $\mu_n^4$ we are going to construct.

    Note that for every $n\io$ we have $S_n=E_{n+1}$ and $\big|S_n\big|=2\big|P_n\big|=2\cdot2^n$, so $\big|S_{n+1}\big|=2\big|S_n\big|$. For every $n\io$ let $c_n=1/2^{n+1}$ and define the auxiliary measure $\nu_n$ as follows:
    \[\nu_n=\sum_{k\in P_n}\alpha_k^n\cdot\big(\delta_{(e_k^n,0)}-\delta_{(o_k^n,0)}\big),\]
    where the coefficients $\alpha_k^n$'s are defined in the following way: for $n=0$ we simply set $\alpha_0^0=1/4$ and for every $n>0$ and $k\in P_n$ we define:
    \[\alpha_k^n=\begin{cases}
    \alpha_{k/2}^{n-1},&\text{ if }e_k^n\in E_{n-1},\\
    c_n/2^n,&\text{ otherwise.}
    \end{cases}\]
    Note that if $e_k^n\in E_{n-1}$, then $k$ is even, so $\alpha_{k/2}^{n-1}$ is well-defined. It also holds $\big|\supp\big(\nu_n\big)\big|=2^{n+1}$.

    It follows that $\big\|\nu_n\big\|=1-c_n$. Indeed, this is obviously true for $n=0$, so fix $n\ge0$ and assume that $\big\|\nu_n\big\|=1-c_n$. Since $E_n\sub S_n=E_{n+1}\sub S_{n+1}$ and $\big|O_{n+1}\big|=\big|E_{n+1}\big|=\big|S_n\big|=2^{n+1}$, we have:
    \[\big\|\nu_{n+1}\big\|=\big\|\nu_n\big\|+2\cdot\big(2^{n+1}/2\big)\cdot\frac{c_{n+1}}{2^{n+1}}=1-c_n+c_{n+1}=1-c_{n+1},\]
    as required.

    We will now show that $\seqn{\nu_n}$ is weak* null. Let $f\in C(K)$ and $\eps>0$. Again, note that $f$ is uniformly continuous, so there is $\delta>0$ such that for every $n\io$ if $1/2^{n+1}<\delta$, then $\big|f\big(e_k^n,0\big)-f\big(o_k^n,0\big)\big|<\eps$. Let thus $N$ be such that $1/2^{n+1}<\delta$ for every $n>N$. We have:
    \[\big|\nu_n(f)\big|\le\sum_{k\in P_n}\alpha_k^n\cdot\big|f\big(e_k^n,0\big)-f\big(o_k^n,0\big)\big|<\eps\cdot\sum_{k\in P_n}\alpha_k^n<\eps\cdot\big(1-c_n\big)<\eps,\]
    which yields that $\lim_{n\to\infty}\nu_n(f)=0$.

    Finally, for every $n\io$ let
    \[\mu_n^4=c_n\cdot\delta_{(e_0^n,0)}+\nu_n,\]
    so $\mu_n^4\big(\big\{\big(e_0^n,0\big)\big\}\big)=c_n+\alpha_0^n$, and hence $\big\|\mu_n^4\big\|=1$ and $(0,0)\in L\big(\mu_n^4\big)$. Since $\lim_{n\to\infty}c_n=0$, the sequence $\seqn{\mu_n^4}$ is weak* null, too.

    We will now prove (i) and (ii) together. First, notice that $\supp\big(\mu_n^4\big)=S_n\times\{0\}$ for every $n\io$, so
    \[S\big(\mu_n^4\big)=\bigcup_{n\io}S_n\times\{0\}=\big\{k/2^{n+1}\colon\ k,n\io,\ 0\le k<2^{n+1}\big\}\times\{0\}.\]
    Next, if for $x\in(0,1]$ and $n\io$ it holds that $x\in S_n$, then $x\in E_{n+1}$, so $\mu_l^4\big(\{(x,0)\}\big)=\alpha_k^{n+1}$ for some $k\in P_{n+1}$ and every $l>n+1$. It follows that $(x,0)\in L\big(\mu_n^4\big)$. 
    (i) and (ii) are thus proved.

    (iii) follows from (ii).
\end{proof}

Let us note here that we presented the constructions of the sequences in Proposition \ref{prop:square_fsjnseqs} in the square $[0,1]^2$ only for simplicity---similar constructions may be carried out also in the unit interval $[0,1]$ or, in fact, any uncountable metric compact space. Note also that the constructed examples satisfy only 4 out of 15 possible relations between the sets $\emptyset$, $L$, $LI$, $LS$, and $S$, however the lacking 11 can be obtained in a similar elementary way.


\medskip

The next lemma shows that the value $1/2$ in the property (iii) of $\seqn{\mu_n^3}$ is not accidental. An intuitive meaning of the lemma is that if for some fixed points of the space $X$ the absolute values of measures of the corresponding singletons grow too much, then they must be nullified by the values on some other points which lie closer and closer to these fixed ones (in the sense of the topology of $X$), cf. also Lemma \ref{lemma:fsjnseq_liminf_s_ls_infinite}. The property (iv) of $\seqn{\mu_n^2}$ implies that we cannot relax here limits to inferior limits or superior limits.

\begin{lemma}\label{lemma:fsjn_sum_lim_12}
For every JN-sequence $\seqn{\mu_n}$ on a space $X$ it holds:
\[\sum_{x\in L(\mu_n)}\lim_{n\to\infty}\big|\mu_n(\{x\})\big|\le 1/2.\]
\end{lemma}
\begin{proof}
Let $\seqn{\mu_n}$ be a JN-sequence on a space $X$. For $x\in L(\mu_n)$ we denote $\lim_{n\to\infty}\mu_n(\{x\})$ by $\mu(x)$.

For the sake of contradiction, assume that $\sum_{x\in L(\mu_n)}\big|\mu(x)\big|>1/2$, 
so there is a finite set $F\sub L\big(\mu_n\big)$ such that $\sum_{x\in F}\big|\mu(x)\big|>1/2$. 
Denote the latter sum by $\alpha$, so $\alpha>1/2$. Let $\eps=\big(\alpha-1/2\big)/2$, so $\alpha=2\eps+1/2$. Let $N\io$ be such that, for every $x\in F$ and every $n>N$, we have:
\[\big|\mu_n(\{x\})-\mu(x)\big|<\eps/|F|.\]

Since $X$ is Tychonoff, we can find a function $f\in C(X,[-1,1])$ such that  $f(x)=\sgn(\mu(x))$ for every $x\in F$  (so $\|f\|_\infty\le1$). For every $n>N$ it holds:
\begin{eqnarray*}
\big|\big(\mu_n\rstr F\big)(f)\big| &=& \Big|\sum_{x\in F}\mu_n(\{x\})f(x) - \sum_{x\in F}\mu(x)f(x) + \sum_{x\in F}\mu(x)f(x)\big| \\ 
&\ge& 
\Big|\sum_{x\in F}\mu(x)f(x)\Big| - \Big|\sum_{x\in F}\big(\mu_n(\{x\}) - \mu(x)\big)f(x)\Big| \\
&\ge& \sum_{x\in F}\big|\mu(x)\big| - \sum_{x\in F}\big|\mu_n(\{x\}) - \mu(x)\big|\big|f(x)\big| \\
&>& \alpha-|F|\cdot\eps/|F|\cdot \|f\|_\infty = \alpha-\eps=\eps+1/2.
\end{eqnarray*}
A similar argument also shows that $\big\|\mu_n\rstr F\big\|>\eps+1/2$, so $\big\|\mu_n\rstr(X\sm F)\big\|<1/2-\eps$.

It follows that for every $n>N$ we have:
\begin{eqnarray*}
\big|\mu_n(f)\big|&=& \big|\big(\mu_n\rstr F\big)(f)+\big(\mu_n\rstr(X\sm F)\big)(f)\big| \ge \big|\big(\mu_n\rstr F\big)(f)\big|-\big|\big(\mu_n\rstr(X\sm F)\big)(f)\big| \\
&>& \eps+1/2 - \|f\|_\infty\cdot\big\|\mu_n\rstr(X\sm F)\big\| > \eps+1/2-1\cdot(1/2-\eps)=2\eps>0,
\end{eqnarray*}

so $\limsup_{n\to\infty}\big|\mu_n(f)\big|>2\eps>0$, which is a contradiction.
\end{proof}



\begin{lemma}\label{lemma:fsjnseq_liminf_s_ls_infinite}
For every pointwise convergent JN-sequence $\seqn{\mu_n}$ on a space $X$, if $\liminf_{k\to\infty}\big\|\mu_k\rstr L\big(\mu_n\big)\big\|<1$, then the set $S\big(\mu_n\big)\sm L\big(\mu_n\big)$ is infinite.
\end{lemma}
\begin{proof}
Let $\seqk{\mu_{n_k}}$ be such a subsequence that $\lim_{k\to\infty}\big\|\mu_{n_k}\rstr L\big\|=\alpha$, where $\alpha<1$. There is $K\io$ such that for every $k>K$ we have:
\[\Big|\big\|\mu_{n_k}\rstr L\big\|-\alpha\Big|<(1-\alpha)/2,\]
so
\[\big\|\mu_{n_k}\rstr L\big\|-\alpha/2<1/2.\]
Since $\seqk{\mu_{n_k}}$ is pointwise convergent, $\lim_{k\to\infty}\mu_{n_k}(\{x\})=0$ for every $x\in S\sm L$, so if $S\sm L$ is finite, then there is $K'>K$ such that for every $k>K'$ we have: 
\[\big\|\mu_{n_k}\rstr(S\sm L)\big\|<(1-\alpha)/2,\]
so 
\[\big\|\mu_{n_k}\rstr(S\sm L)\big\|+\alpha/2<1/2,\] 
but then for every $k>K'$ we also have:
\[1=\big\|\mu_{n_k}\big\|=\Big(\big\|\mu_{n_k}\rstr(S\sm L)\big\|+\alpha/2\Big)+\Big(\big\|\mu_{n_k}\rstr L\big\|-\alpha/2\Big)<1/2+1/2=1,\]
a contradiction.
\end{proof}

Note that Proposition \ref{prop:square_fsjnseqs}.(4) provides an example of a JN-sequence for which both the assumption as well as the conclusion stated in the above lemma do not hold. 

The following lemma asserts an interesting and useful property of the subspace $S\big(\mu_n\big)$ for a given JN-sequence $\seqn{\mu_n}$. Its proof is contained in the proof of \cite[Theorem 1]{BKS19}, so we skip it; another proof was also given in \cite[Proposition 4.1]{KMSZ}.

\begin{lemma}\label{lemma:fsjn_f_bounded}
Let $\seqn{\mu_n}$ be a JN-sequence on a Tychonoff space $X$. Then, every function $f\in C(X)$ is bounded on the subspace $\ol{S\big(\mu_n\big)}^X$.\noproof
\end{lemma}

\begin{corollary}
If a normal space $X$ admits a JN-sequence $\seqn{\mu_n}$, then the subspace $\ol{S\big(\mu_n\big)}^X$ is pseudocompact.
\end{corollary}
\begin{proof}
Put $S=\ol{S\big(\mu_n\big)}^X$. Let $f\in C(S)$. By the Tietze extension theorem there is $F\in C(X)$ extending $f$. By Lemma \ref{lemma:fsjn_f_bounded}, $f=F\rstr S$ is bounded.
\end{proof}

The following criterion for a sequence of measures to be a JN-sequence will be useful in the sequel.

\begin{lemma}\label{lemma:norm_conv_fsjn_seqs}
Let $\seqn{\mu_n}$ and $\seqn{\nu_n}$ be two finitely supported sequences of measures on a space $X$ such that $\lim_{n\to\infty}\big\|\mu_n-\nu_n\big\|=0$. Assume that $\seqn{\mu_n}$ is a JN-sequence on $X$, $\big\|\nu_n\big\|=1 $ for every $n\io$, and that every function $f\in C(X)$ is bounded on $S\big(\nu_n\big)$. Then, $\seqn{\nu_n}$ is also a JN-sequence on $X$.
\end{lemma}
\begin{proof}
It is only necessary to prove that $\lim_{n\to\infty}\nu_n(f)=0$ for every $f\in C(X)$. Let thus $f\in C(X)$ and put $\alpha=\sup\big\{|f(x)|\colon\ x\in S\big(\mu_n\big)\big\}$ and $\beta=\sup\big\{|f(x)|\colon\ x\in S\big(\nu_n\big)\big\}$. By Lemma \ref{lemma:fsjn_f_bounded} the function $f$ is bounded on $S\big(\mu_n\big)$, so $\alpha<\infty$. Similarly, $\beta<\infty$ by the assumption. We then have:
\[\big|\nu_n(f)\big|\le\big|\mu_n(f)-\nu_n(f)\big|+\big|\mu_n(f)\big|\le\max(\alpha,\beta)\cdot\big\|\mu_n-\nu_n\big\|+\big|\mu_n(f)\big|,\]
so $\lim_{n\to\infty}\nu_n(f)=0$. It follows that $\seqn{\nu_n}$ is a JN-sequence on $X$.
\end{proof}

Combining Lemmas \ref{lemma:jnseq_pos_neg} and \ref{lemma:norm_conv_fsjn_seqs} we easily get the following corollary.

\begin{corollary}
If $\seqn{\mu_n}$ is a JN-sequence on a space $X$, then there are a subsequence $\seqk{\mu_{n_k}}$ and a JN-sequence $\seqk{\nu_k}$ on $X$ such that $\supp\big(\nu_k\big)=\supp\big(\mu_{n_k}\big)$ and
\[\big\|\nu_k\rstr P_k\big\|=\big\|\nu_k\rstr N_k\big\|=1/2\]
for every $k\io$, where the sets $P_k$'s and $N_k$'s are defined for $\seqk{\nu_k}$ similarly as in Lemma \ref{lemma:jnseq_pos_neg}.

\noproof
\end{corollary}


\section{Disjointly supported JN-sequences\label{section:fsjn_disjoint_supps}}

In this section we will show that if a compact space $K$ admits a JN-sequence, then $K$ carries a JN-sequence with disjoint supports (Theorem \ref{theorem:disjoint_supps}). Let us thus start with the following convenient definition.

\begin{definition}
A finitely supported sequence $\seqn{\mu_n}$ of measures on a space $X$ is \textit{disjointly supported} if $\supp\big(\mu_n\big)\cap\supp\big(\mu_{n'}\big)=\emptyset$ for every $n\neq n'\io$.
\end{definition}

%
%
The next lemma is an easy consequence of Kadec--Pe\l czy\'{n}ski--Rosenthal's Subsequence Splitting Lemma (cf. \cite[Lemma 5.2.8]{AK06}) and Schur's property of the Banach space $\ell_1$. However, for the sake of completeness we include a short elementary proof of it, directly stated in terms of measures. Recall that a sequence $\seqn{\mu_n}$ of Borel measures on a space $X$ \textit{converges in norm} to a Borel measure $\mu$ if $\big\|\mu_n-\mu\big\|=\big|\mu_n-\mu\big|(X)\to0$ as $n\to\infty$.

\begin{lemma}\label{lemma: KPR splitting lemma}
For every sequence $\seqn{\mu_n}$ of finitely supported measures on a space $X$ which is bounded, i.e. there is $M>0$ such that $\big\|\mu_n\big\|<M$ for every $n\io$, there exists a subsequence $\seqk{\mu_{n_k}}$ and a sequence $\seqk{A_k}$ of pairwise disjoint finite subsets of $S=S\big(\seqn{\mu_n}\big)$ such that the sequence $\seqk{\mu_{n_k}\rstr \big(S \setminus A_k\big)}$ converges in norm to a measure $\mu$ on $X$ such that $\mu=\sum_{x\in S}\alpha_x\cdot\delta_x$, for some $\alpha_x\in\R$ ($x\in S$), and $\|\mu\|\le M$.
\end{lemma}

\begin{proof}
Since the sequence $\seqn{\mu_n}$ is bounded,  without loss of generality, we can assume that $\seqn{\mu_n}$ is pointwise convergent to a measure $\mu$ on $X$ such that $\mu=\sum_{x\in S}\alpha_x\cdot\delta_x$ with $\|\mu\|=\sum_{x\in S}\big|\alpha_x\big|\le M$, cf. Remark \ref{remark:jn_pointwise_limits}. By induction we will choose an increasing sequence $\seqn{n_k}$ and a sequence  of pairwise disjoint finite sets $A_k \subseteq S$ such that for every $k\io$ we have:
\[\tag{$*$}\big\|\mu - \big(\mu_{n_k}\rstr \big(S\setminus A_k\big)\big)\big\| < 1/(k+1),\]
which will mean that $\seqk{\mu_{n_k}\rstr \big(S\setminus A_k\big)}$ converges in norm to $\mu$.


Let $k\io$ and assume that we have constructed sequences $n_0,\ldots,n_{k-1}$ and $A_0,\ldots,A_{k-1}$ as required. At the stage $k$, we take a finite set $B_k\subseteq S$ such that $\bigcup_{j=0}^{k-1} A_j  \subseteq B_k$ and
\[\big\|\mu\rstr \big(S \setminus B_k\big)\big\|  < 1/(2k+2).\]
Using the pointwise convergence of  $\seqn{\mu_n}$ we can pick $n_k>n_{k-1}$ (where $n_{-1}=-1$) such that
\[\big\|\mu\rstr B_k - \mu_{n_k}\rstr B_k\big\| < 1/(2k+2).\]
Finally, for $A_{k} = \supp\big(\mu_{n_k}\big) \setminus B_k$ we have:
\begin{align*}
\big\|\mu-\big(\mu_{n_k}\rstr \big(S\setminus A_k\big)\big)\big\|&\le\big\|\mu\rstr\big(S\sm B_k\big)\big\|+\big\|\mu\rstr B_k-\mu_{n_k}\rstr B_k\big\|+\big\|\mu_{n_k}\rstr B_k-\mu_{n_k}\rstr\big(S\sm A_k\big)\big\|\\
&=\big\|\mu\rstr\big(S\sm B_k\big)\big\|+\big\|\mu\rstr B_k-\mu_{n_k}\rstr B_k\big\|+0<2/(2k+2)=1/(k+1),
\end{align*}
which gives ($*$) for $k+1$.
\end{proof}

From the above lemma we can derive the main result of this section.

\begin{theorem}\label{theorem:disjoint_supps}
Assume that a space $X$ carries a JN-sequence. Then, $X$ admits a disjointly supported JN-sequence.
\end{theorem}

\begin{proof}
Let $\seqn{\mu_n}$ be a JN-sequence on a space $X$, and set $S = S\big(\mu_n\big)$. Let $\seqk{\mu_{n_k}}$, $\seqk{A_k}$, and $\mu$ be as in Lemma \ref{lemma: KPR splitting lemma}. For every $k\io$, we set $\nu^1_k=\mu_{n_k}\rstr A_k$ and $\nu^2_k=\mu_{n_k}\rstr \big(S \setminus A_k\big)$; so, $\mu_{n_k} = \nu^1_k + \nu^2_k$. Obviously, $\seqk{\nu^1_k}$ is disjointly supported.

Observe that for some $\varepsilon > 0$ we have $\big\|\nu^1_k\big\|> \varepsilon$ for all $k\in\omega$. Otherwise, we would have a subsequence $\seqi{\nu^1_{k_i}}$ converging in norm to the zero measure. Then, $\seqi{\mu_{n_{k_i}}}$ would converge in norm to $\mu$, so $\|\mu\| = 1$. For every $f\in C(X)$, $f\rstr S$ is a bounded function by Lemma \ref{lemma:fsjn_f_bounded}, so
\[\mu(f)=\sum_{x\in S}f(x)\mu(\{x\})=0,\]
since $\lim_{i\to\infty}\mu_{n_{k_i}}(f) = 0$ and
\begin{align*}
\big|\mu_{n_{k_i}}(f)-\mu(f)\big|&=\big|\sum_{x\in S}f(x)\big(\mu_{n_{k_i}}(\{x\})-\mu(\{x\})\big)\big|\le\|f\rstr S\|_\infty\cdot\sum_{x\in S}\big|\mu_{n_{k_i}}(\{x\})-\mu(\{x\})\big|\\
&=\|f\rstr S\|_\infty\cdot\big\|\mu_{n_{k_i}}-\mu\big\|\xrightarrow{i\to\infty}0.
\end{align*}
On the other hand, for a finite $A\subseteq S$ and $\delta>0$ such that $\|\mu\rstr A\|=1/2+\delta$ (recall that $\|\mu\|=1$), and a function $f\in C(X,[-1,1])$ such that $f(x)=\sgn(\mu(x))$ for every $x\in A$ (so $\|f\|_\infty=1$), we would have:
\begin{align*}
|\mu(f)|&=|\mu(f\rstr A)+\mu(f\rstr(S\sm A))|\ge|\mu(f\rstr A)|-|\mu(f\rstr(S\sm A))|\\
&=\|\mu\rstr A\|-|\mu(f\rstr(S\sm A))|\ge1/2+\delta-\|f\rstr S\|_\infty\cdot\|\mu\rstr(S\sm A)\|\\
&=1/2+\delta-1/2+\delta=2\delta > 0,
\end{align*}
a contradiction.

Since $\seqk{\mu_{n_k}}$ is a JN-sequence and $\seqk{\nu^2_k}$ converges in norm (and hence weak*) to $\mu$, $\seqk{\nu^1_k}$ is weak* convergent to $-\mu$. Let $\rho_k = \nu^1_{2k} - \nu^1_{2k+1}$ for $k\in \omega$. It follows that $\seqk{\rho_k}$ is weak* null. Clearly, the supports of $\rho_k$'s are pairwise disjoint. Since $\big\|\rho_k\big\| > 2\varepsilon$ for every $k\io$, the sequence $\seqk{\rho_k\big/\big\|\rho_k\big\|}$ is the desired disjointly supported JN-sequence.
\end{proof}

\section{JN-sequences with discrete union of supports\label{section:discrete}}

In this section we prove that if a given Tychonoff space $X$ admits a JN-sequence, then it admits one with disjoint supports whose union is discrete. 

The following lemma and corollary are a simple application of the triangle inequality.

\begin{lemma}\label{lem:subseq}
Let $\seqn{\mu_n}$ be a JN-sequence on a space $X$. Let $U_1,\ldots,U_m$ ($m\io$) be pairwise disjoint subsets of $X$. Then, there exists $1\le i\le m$ and a strictly increasing subsequence $\seqk{n_k}$ such that $\big|\mu_{n_k}\big|\big(U_i\big)\le 1/m$ for every $k\io$.\noproof
\end{lemma}

\begin{corollary}\label{cor:subseq}
Let $\seqn{\mu_n}$ be a disjointly supported JN-sequence on a space $X$. For every $\eps>0$ there exist $n_0\io$, an open subset $U\sub X$, and a strictly increasing subsequence $\seq{n_k}{k\ge 1}$ such that
\begin{itemize}
	\item $\supp\big(\mu_{n_0}\big)\sub U$, 
	\item $n_1>n_0$, and
	\item $\big|\mu_{n_k}\big|(U)<\eps$ for every $k\ge 1$.\noproof
\end{itemize}
\end{corollary}

In order to prove the next lemma, which constitutes the core of the proof of Theorem \ref{theorem:discrete}, we need the following family of auxiliary functions: for each $a,b\in\Q$ such that $0\le a<b\le 1$ define the  continuous piecewise linear function $p_{a,b}\colon[0,1]\to[0,1]$ by the formula:
\[p_{a,b}(t)=
\begin{cases}
	0,&\text{ if }t\le a,\\
	\frac{t-a}{b-a},& \text{ if }t\in(a,b),\\
	1,&\text{ if }t\ge b,
\end{cases}\]
where $t\in [0,1]$.

\begin{lemma}\label{lem:main}
Let $X$ be a space and $\seqk{\nu_k}$ a disjointly supported JN-sequence on $X$. For every $k\io$ set $E_k=\supp\big(\nu_k\big)$. Let $\seqk{U_k}$ be a sequence of open subsets of $X$ such that $E_k\sub U_k$ for every $k\io$ and $\big|\nu_l\big|\big(U_0)<1/4$ for every $l>0$. Then, there exist a strictly increasing sequence $\seqi{k_i}$ with $k_0=0$ and a sequence $\seqi{g_i}$ of functions in $C(X,[0,1])$ such that for every $i\io$ the following conditions are satisfied:
	\[\tag{L.1}\bigcup_{j=0}^i E_{k_j}\sub\intt g_i^{-1}(0),\]
	\[\tag{L.2}g_i\rstr\Big(X\sm\bigcup_{j=0}^iU_{k_j}\Big)\equiv 1,\]
	\[\tag{L.3}g_i^{-1}(0)\sub g_{i+1}^{-1}(0)\quad\text{and}\quad g_{i+1}^{-1}(0)\sub g_i^{-1}\big[[0,1)\big]\cup U_{k_{i+1}},\]
	\[\tag{L.4}\text{the set }A_i=\big\{l\in A_{i-1}\sm\{k_i\}\colon\ \big|\nu_l\big|\big(g_i^{-1}\big[(0,1)\big]\big)<1/2^i\big\}\text{ is infinite (where }A_{-1}=\omega\text{)},\]
	\[\text{and }k_{i+1}=\min A_i.\]
\end{lemma}
\begin{proof}
Since $X$ is Tychonoff, for every $k\io$ there exists $h_k'\in C(X,[0,1])$ such that $h_k'\rstr E_k\equiv 0$ and $h_k'\rstr\big(X\sm U_k\big)\equiv 1$. Set $h_k=p_{\frac{1}{2},1}\circ h_k'$. Obviously, $h_k\in C(X,[0,1])$, too, and
\[\tag{P.1}E_k\sub\intt h_k^{-1}(0)\quad\text{and}\quad h_k\rstr\big(X\sm U_k\big)\equiv 1.\]

Let $k_0=0$ and $g_0=h_0$---conditions (L.1) and (L.2) are trivially satisfied by (P.1). Set $A_{-1} = \omega$ (so $k_0=\min A_{-1}$). Since
\[g_0^{-1}\big[(0,1)\big]\sub U_0,\]
the assumption on $U_0$ implies that for every $l>0$ we have
\[\big|\nu_l\big|\big(g_0^{-1}\big[(0,1)\big]\big)<1/4,\]
that is, that $A_0=\{l>0\colon l\io\}$ and hence that $A_0$ is infinite.

Fix $i\io$ and let us assume that sequences $k_0<k_1<\ldots<k_i$, $g_0,g_1,\ldots,g_i$ and $A_0,A_1,\ldots,A_i$ satisfying conditions (L.1)--(L.4) have been constructed. Set
\[k_{i+1}=\min A_i\]
and
\[g_{i+1}'=\min\big(g_i,h_{k_{i+1}}\big).\]
Of course, $g_{i+1}'\in C(X,[0,1])$. From conditions (P.1), (L.1), and (L.2) we conclude that
\[\tag{P.2}\bigcup_{j=0}^{i+1}E_{k_j}\sub\intt\big(g_{i+1}'\big)^{-1}(0)\quad\text{and}\quad g_{i+1}'\rstr\Big(X\sm\bigcup_{j=0}^{i+1}U_{k_j}\Big)\equiv 1.\]
By Lemma \ref{lem:subseq}, there is $m\in\big\{0,1,\ldots,2^{i+1}\big\}$ such that the following set
\[B_{i+1}=\Big\{l>k_{i+1}\colon\ \big|\nu_l\big|\big(\big(g_{i+1}'\big)^{-1}\Big[\Big(\frac{m}{2^{i+1}+1},\frac{m+1}{2^{i+1}+1}\Big)\Big]\big)<1/2^{i+1}\Big\}\]
is infinite. We finally define:
\[g_{i+1}=p_{m/(2^{i+1}+1),(m+1)/(2^{i+1}+1)}\circ g_{i+1}'.\]
As always, $g_{i+1}\in C(X,[0,1])$. We also have:
\[g_{i+1}^{-1}[(0,1)]=\big(g_{i+1}'\big)^{-1}\Big[\Big(\frac{m}{2^{i+1}+1},\frac{m+1}{2^{i+1}+1}\Big)\Big],\]
so $A_{i+1}=B_{i+1}$, and hence $A_{i+1}$ is infinite. Condition (L.4) is thus satisfied for $i+1$ (with $k_{i+2}=\min A_{i+1}$). For $s\in\{0,1\}$ we have:
\[\big(g_{i+1}'\big)^{-1}(s)\sub g_{i+1}^{-1}(s),\]
so condition (P.2) implies conditions (L.1) and (L.2) for $i+1$. Also, since $g_{i+1}'\le g_i$ and so
\[g_i^{-1}(0)\sub\big(g_{i+1}'\big)^{-1}(0)\sub g_{i+1}^{-1}(0),\]
we get the first part of (L.3) for $i+1$. From condition (P.1) and the construction of $g_{i+1}$ we get that
\[g_{i+1}\rstr\Big(X\sm\big(g_i^{-1}\big[[0,1)\big]\cup U_{k_{i+1}}\big)\Big)\equiv 1,\]
and hence the second part of (L.3) holds for $i+1$, too. The induction is thus finished.
\end{proof}

(Note that in the proofs of the above three results we do not make any use of the assumption that $\seqn{\mu_n}$ converges to $0$ on continuous functions.)

We are in the position to prove the main theorem. Recall that by Theorem \ref{theorem:disjoint_supps} if a space admits a JN-sequence, then it carries a disjointly supported JN-sequence.

\begin{theorem}\label{theorem:discrete}
Let $X$ be a space and $\seqn{\mu_n}$ a disjointly supported JN-sequence on $X$. Then, there exist a disjointly supported JN-sequence $\seqi{\rho_i}$ on $X$ such that the union $\bigcup_{i\in\omega}\supp\big(\rho_i\big)$ is a discrete subset of $X$, and a subsequence $\seqi{n_i}$ such that $\supp\big(\rho_i\big)\sub\supp\big(\mu_{n_i}\big)$ for every $i\io$.
\end{theorem}
\begin{proof}
Using inductively Corollary \ref{cor:subseq}, we find a strictly increasing sequence $\seqk{n_k}$ and a sequence $\seqk{U_k}$ of (not necessarily pairwise disjoint) open subsets of $X$ such that for every $k\io$ we have $\supp\big(\mu_{n_k}\big)\sub U_k$ and
\[\tag{Q.1}\big|\mu_{n_i}\big|\big(U_k\big)<\frac{1}{4}\cdot\frac{1}{2^k}\]
for every $i>k$. For every $k\io$ set $E_k=\supp\big(\mu_{n_k}\big)$ and $\nu_k=\mu_{n_k}$. Let sequences $\seqi{k_i}$ and $\seqi{g_i}$ be as in Lemma \ref{lem:main}.

For every $i\io$ put $C_i=g_i^{-1}(0)$ and notice that by condition (L.2) we have
\[C_i\sub\bigcup_{j=0}^i U_{k_j},\]
so condition (Q.1) gives us that
\[\tag{Q.2}\big|\nu_{k_l}\big|\big(C_i\big)<1/2\]
for every $l>i$. Condition (L.3) implies for every $i\io$ that
\[\tag{Q.3}C_i\sub C_{i+1}\quad\text{and}\quad C_{i+1}\sm C_i\sub g_i^{-1}[(0,1)]\cup U_{k_{i+1}},\]
hence, by conditions (L.4) and (Q.1),
\[\tag{Q.4}\big|\nu_{k_l}\big|\big(C_{i+1}\sm C_i\big)<\frac{1}{2^i}+\frac{1}{4}\cdot\frac{1}{2^{k_{i+1}}}<\frac{2}{2^i}\]
for every $l>i+1$.

Assuming that $C_{-1}=\emptyset$, for every $i\io$ define the measure $\lambda_i$ on $X$ by the formula:
\[\lambda_i=\nu_{k_i}\rstr\big(X\sm C_{i-1}\big).\]
It follows that $\supp\big(\lambda_i\big)\sub\supp\big(\nu_{k_i}\big)$ and condition (Q.2) implies that $\big\|\lambda_i\big\|>1/2$.

We now show that $\lim_{i\to\infty}\lambda_i(f)=0$ for every $f\in C(X)$. So let us fix $f\in C(X)$ and $\eps>0$. By Lemma \ref{lemma:fsjn_f_bounded}, there is $M>0$ such that $|f(x)|\le M$ for every $x\in S\big(\seqi{\nu_{k_i}}\big)$. Let $m\io$ be such that $\sum_{i=m}^\infty1/2^i<\eps$. For every $x\in X$ set $g(x)=g_m(x)\cdot f(x)$, so $g\in C(X,[0,1])$. Conditions (L.4) and (Q.3) and the definitions of $C_m$ and $\lambda_l$ imply that for every $l>m$ we have:
\begin{align*}\tag{Q.5}
\big|\lambda_l(f)-\lambda_l(g)\big|&\le\big|\int_{C_m}(f-g)\der\lambda_l\big|+\big|\int_{g_m^{-1}[(0,1)]}(f-g)\der\lambda_l\big|+\big|\int_{g_m^{-1}(1)}(f-g)\der\lambda_l\big|\\
&\le0 + M\cdot\frac{1}{2^m} + 0 = M/2^m<M\cdot\eps.
\end{align*}
Since $\lim_{i\to\infty}\nu_{k_i}(g)=0$ (as $\seqi{\nu_{k_i}}$ is a subsequence of the JN-sequence $\seqn{\mu_n}$), there is $n\io$ such that 
\[\tag{Q.6}\big|\nu_{k_l}(g)\big|<\eps\]
for every $l>n$. For $l>m+1$, by (Q.4) we have:
\begin{align*}\tag{Q.7}
&\big|\lambda_l(g)-\nu_{k_l}(g)\big|\\
&\le\big|\int_{C_m}g\,\der\big(\lambda_l-\nu_{k_l}\big)\big|+\sum_{i=m}^{l-2}\big|\int_{C_{i+1}\sm C_i}g\,\der\big(\lambda_l-\nu_{k_l}\big)\big|+\big|\int_{X\sm C_{l-1}}g\,\der\big(\lambda_l-\nu_{k_l}\big)\big|\\
&\le0+\sum_{i=m}^{l-2}M\cdot\big|\nu_{k_l}\big|\big(C_{i+1}\sm C_i\big)+0\le M\cdot\sum_{i=m}^{l-2}\frac{2}{2^i}<2M\eps.
\end{align*}
Finally, for every $l>\max(n,m+1)$, by conditions (Q.5)--(Q.7), we get that 
\[\big|\lambda_l(f)\big|<\eps(1+3M),\]
which implies that $\lim_{i\to\infty}\lambda_i(f)=0$.

The set $S\big(\lambda_i\big)=\bigcup_{i\io}\supp\big(\lambda_i\big)$ is a discrete subset of $X$, because the supports are pairwise disjoint and for every $i\io$ we have $\bigcup_{j=0}^i\supp\big(\lambda_j\big)\sub\intt C_i$ (by condition (L.1)) and for every $j>i$ the support $\supp\big(\lambda_j\big)$ is contained in the open set $X\sm C_i$ (by condition (L.3)). It follows that the sequence $\seqi{\rho_i}$ of measures on $X$ defined for every $i\io$ by the formula
\[\rho_i=\lambda_i\big/\big\|\lambda_i\big\|\]
is a JN-sequence on $X$ such that the set $S\big(\rho_i\big)=\bigcup_{i\io}\supp\big(\rho_i\big)$ is a discrete subset of $X$ and $\supp\big(\rho_i\big)\sub\supp\big(\mu_{n_{k_i}}\big)$ for every $i\io$.
\end{proof}

\medskip

\begin{proof}[Proof of Theorem {\ref{theorem:main}}]
Combine Theorems \ref{theorem:disjoint_supps} and \ref{theorem:discrete}.
\end{proof}

\begin{proof}[Proof of Corollary {\ref{cor:main}}]
Combine Theorem \ref{theorem:main} and the Tietze extension theorem.
\end{proof}

\medskip

The following corollary is also an immediate consequence of Theorem \ref{theorem:discrete}. Of course, it is true also for every space $X$ with a base consisting of clopen subsets.

\begin{corollary}
Let $X$ be a totally disconnected compact space carrying a JN-sequence. Then, there exist a JN-sequence $\seqn{\mu_n}$ on $X$ and a sequence $\seqn{U_n}$ of pairwise disjoint clopen subsets of $X$ such that $\supp\big(\mu_n\big)\sub U_n$ for every $n\io$.\noproof
\end{corollary}



\section{Sizes of supports in JN-sequences \label{section:sizes_of_supports}}

In this section we will study possible cardinalities of supports of measures from JN-sequences. We have two cases here: either (1) a space $X$ admits a JN-sequence $\seqn{\mu_n}$ for which there exists $M\io$ such that $\big|\supp\big(\mu_n\big)\big|\le M$ for every $n\io$, or (2) every JN-sequence $\seqn{\mu_n}$ on $X$ has the property that $\lim_{n\to\infty}\big|\supp\big(\mu_n\big)\big|=\infty$. As an example of the former case we may name any space $X$ having a non-trivial convergent sequence. An appropriate example for the latter case is more difficult to find---however, it appears that the space $K$ considered in \cite[Section 4]{BKS19} (\textit{Plebanek's example}) has the required property. In Subsection \ref{section:sizes_two_examples} we prove this statement as well as we present another example (investigated by Bereznitski\u{\i} and Schachermayer) which is in many aspects very similar to Plebanek's one but satisfies the case (1).

In Subsection \ref{section:study_of_sizes_of_supports} we will provide several general statements concerning cardinalities of supports. In particular, we prove in Theorem \ref{theorem:sizes_of_supps} that if a  space $X$ satisfies the case (1), then there exists a JN-sequence $\seqn{\mu_n}$ such that $\big|\supp\big(\mu_n\big)\big|=2$ for every $n\io$.

\subsection{Two examples}\label{section:sizes_two_examples}

We first recall some standard notions. For a Boolean algebra $\aA$ by $St(\aA)$ we denote its Stone space. Recall that $St(\aA)$ is a totally disconnected compact space and that the Boolean algebra of clopen subsets of $St(\aA)$ is isomorphic to $\aA$. For every element $A\iA$ by $[A]_\aA$ we denote the corresponding clopen subset of $St(\aA)$. 

Recall also that, by the Stone--Weierstrass theorem, a finitely supported sequence $\seqn{\mu_n}$ of measures on a totally disconnected compact space $K$ (or, equivalently, on the Stone space $St(\aA)$ of some Boolean algebra $\aA$) is weak* null if and only if $\lim_{n\to\infty}\mu_n(U)=0$ for every clopen set $U\sub K$.

\begin{example}\label{example:plebanek}
In \cite[Section 4]{BKS19}, the authors provided the example due to Plebanek which uses the following Boolean algebra $\dD$:
\[\dD=\Big\{A\in\wo\colon\ \lim_{n\to\infty}\frac{|A\cap\{0,\ldots,n-1\}|}{n}\in\{0,1\}\Big\}.\]
Since for each $n\io$ the set $\{n\}$ belongs to $\dD$ and is an atom therein, we may consider $St(\dD)$ as a compactification of $\omega$. Let us additionally define the ideal $\zZ$ and the ultrafilter $p$ in $\dD$ as follows:
\[\zZ=\Big\{A\in\wo\colon\ \lim_{n\to\infty}\frac{|A\cap\{0,\ldots,n-1\}|}{n}=0\Big\}\]
and
\[p=\dD\sm\zZ.\]
\end{example}

We have the following result.

\begin{proposition}\label{prop:ex_plebanek}
The Boolean algebra $\dD$ has the following properties:
\begin{enumerate}
    \item $St(\dD)$ does not have any non-trivial convergent sequences;
    \item if $X\sub St(\dD)$ is infinite, then there exists an infinite subset $Y\sub X$ such that $\ol{Y}^{St(\dD)}$ is homeomorphic to $\bo$;
    \item $St(\dD)$ carries a JN-sequence;
    \item every JN-sequence $\seqn{\mu_n}$ on $St(\dD)$ has the property that $\lim_{n\to\infty}\big|\supp\big(\mu_n\big)\big|=\infty$.
\end{enumerate}
\end{proposition}
\begin{proof}
For (1)--(3), see \cite[Section 4, Fact 1--3, page 3026]{BKS19}. We
now prove (4), so for the sake of contradiction let us assume that
there exists a JN-sequence $\seqn{\mu_n}$ on $St(\dD)$ and  an
integer $M\io$ such that $\big|\supp\big(\mu_n\big)\big|\le M$ for every
$n\io$. By Theorem \ref{theorem:sizes_of_supps}, we may assume that
$\mu_n=\frac{1}{2}\big(\delta_{x_n}-\delta_{y_n}\big)$. By Lemma
\ref{lemma:disjoint_supps_size_2}, we may also assume that
$\big\{x_n,y_n\big\}\cap\big\{x_{n'},y_{n'}\big\}=\emptyset$ for
every $n\neq n'\io$ and that $p\not\in\big\{x_n,y_n\big\}$ for every
$n\io$. We need to consider several cases:
\begin{enumerate}[(i)]
    \item There is $Q\in\cso$ such that $\big\{x_n,y_n\big\}\sub\omega$ for every
     $n\in Q$. We then go to a subsequence $\seqk{n_k\in Q}$ such that $A=\bigcup_{k\io}\big\{x_{n_k},y_{n_k}\big\}\in\zZ$. Since $[A]_\dD$ is homeomorphic to $\bo$, it follows that $\seqk{\mu_{n_k}\rstr[A]_\dD}$ gives rise to a JN-sequence in $\bo$, which is impossible (see \cite{BKS19} or \cite{KSZgroth}).
    \item There is $Q\in\cso$ such that   $\big\{x_n,y_n\big\}\cap\omega=\emptyset$ for
     every $n\in Q$. We find $A_n\in\zZ$ such that $\big\{x_n,y_n\big\}\sub\big[A_n\big]_\dD$
     for every $n\in Q$. By \cite[Section 4, Fact 1, page 3026]{BKS19}, there is infinite
      $B\in\zZ$ such that $A_n\sm B$ is finite for every $n\in Q$. Since
       $\big\{x_n,y_n\big\}\cap\omega=\emptyset$ for every $n\in Q$, it follows that
        $A_n\sm B\not\in x_n$ and $A_n\sm B\not\in y_n$, and hence
        $\big\{x_n,y_n\big\}\sub[B]_\dD$. Again, since $[B]_\dD$ is homeomorphic to $\bo$, we
        obtain a JN-sequence on $\bo$, which is a contradiction.
    \item There is $Q\in\cso$ such that $\big|\big\{x_n,y_n\big\}\cap\omega\big|=1$ for every $n\in Q$. Without loss of generality, we may assume that $x_n\io$ for every $n\in Q$. First, let us find $R\in[Q]^\omega$ such that $\big\{x_n\colon\ n\in R\big\}\in\zZ$. Then, similarly as in (ii), let us find $B\in\zZ$ such that $\big\{y_n\colon\ n\in R\big\}\sub[B]_\dD$. Since $\zZ$ is an ideal, $C=\big\{x_n\colon\ n\in R\big\}\cup B\in\zZ$. It follows that $[C]_\dD$ is homeomorphic to $\bo$ and $\seq{\mu_n\rstr[C]_\dD}{n\in R}$ is a JN-sequence on $[C]_\dD$, a contradiction.
\end{enumerate}
\end{proof}

\begin{example}\label{example:schachermayer}
In \cite{Ber71} Bereznitski\u{\i} investigated properties of the following example $K_B$ of a compact space: Let  $K_B$ be a quotient space obtained by identifying points $(x,0)$ and $(x,1)$ in $\beta\omega \times \{0,1\}$ for all $x\in \beta\omega \setminus \omega$. 

In \cite[Example 4.10]{Sch82} Schachermayer considered the same space $K_B$ described as the Stone space of some simple Boolean algebra (see also \cite[Example 6.9]{MSnf}).
\end{example}

The next proposition shows that the above space $K_B$ and Plebanek's example $St(\dD)$ share  similar properties.

\begin{proposition}\label{prop:ex_schachermayer}
The compact space $K_B$ has the following properties:
\begin{enumerate}
    \item $K_B$ does not have any non-trivial convergent sequences;
    \item if $X\sub K_B$ is infinite, then there exists an infinite subset
    $Y\sub X$ such that $\ol{Y}^{K_B}$ is homeomorphic to $\bo$;
    \item there exists a JN-sequence $\seqn{\mu_n}$ on $K_B$ such that $\big|\supp\big(\mu_n\big)\big|=2$ for every $n\io$.
\end{enumerate}
\end{proposition}
\begin{proof}
Properties (1) and (2) follow immediately from the fact that $K_B$ is a union of two copies of $\bo$, and the space $\bo$ has these properties, cf.\ \cite[Chapter 3.6]{En}.

Property (3) is witnessed by the following JN-sequence:
\[\mu_n={\textstyle\frac{1}{2}}\big(\delta_{q((n,0))}-\delta_{q((n,1))}\big),\ n\io,\]
where $q\colon\beta\omega \times \{0,1\} \to K_B$ is the quotient map. Note that for every clopen  $U\sub K_B$ and for all but finitely many $n\io$ we have $\big\{q((n,0)),q((n,1))\big\} \sub U$ or $\big\{q((n,0)),q((n,1))\big\} \cap U = \emptyset$ .
\end{proof}

\subsection{Estimations of sizes of supports\label{section:study_of_sizes_of_supports}}

We will now restrict our study to those spaces which admit JN-sequences with bounded sizes of supports, i.e. such JN-sequences $\seqn{\mu_n}$ that there exists $M\io$ such that $\big|\supp\big(\mu_n\big)\big|\le M$ for every $n\io$. We start with two simple lemmas.

\begin{lemma}\label{lemma:fsjn_conv_seq}
Let $X$ be a space. Fix a sequence $\seqn{x_n}$ in $X$ and a point $x\in X$. For every $n\io$ put $\mu_n=\frac{1}{2}\big(\delta_{x_n}-\delta_x\big)$. Then, $\seqn{\mu_n}$ is a JN-sequence if and only if $x_n\to x$ in $X$.\noproof
\end{lemma}

\begin{lemma}\label{lemma:disjoint_supps_size_2}
Let a space $X$ admit a JN-sequence $\seqn{\mu_n}$ defined for every $n\io$ as $\mu_n=\frac{1}{2}\big(\delta_{x_n}-\delta_{y_n}\big)$, where $x_n,y_n\in X$. Then, there exists a disjointly supported JN-sequence $\seqn{\nu_n}$ defined for every $n\io$ as $\nu_n=\frac{1}{2}\big(\delta_{u_n}-\delta_{w_n}\big)$, where $u_n,w_n\in X$.
\end{lemma}

\begin{proof}
If the space $X$ contains a non-trivial convergent sequence $\seqn{z_n}$, then it is easy to see that the measures defined as $\nu_n=\frac{1}{2}\big(\delta_{z_{2n}}-\delta_{z_{2n+1}}\big)$ form a JN-sequence satisfying the conclusion of the lemma.

If $X$ does not contain any non-trivial convergent sequences, then, by Lemma \ref{lemma:fsjn_conv_seq}, for every $A\in\cso$ we have $\bigcap_{n\in A}\supp\big(\mu_n\big)=\emptyset$, so there exists a subsequence $\seqk{\mu_{n_k}}$ such that $\supp\big(\mu_{n_k}\big)\cap\supp\big(\mu_{n_l}\big)=\emptyset$ for every $k\neq l\io$. To finish the proof put $\nu_k=\mu_{n_k}$ for every $k\io$.
\end{proof}

The next observation follows immediately from the definition of a JN-sequence applied for the constant function $1_X$ on $X$. 

\begin{lemma}\label{lemma:supports_never_singletons}
Let $\seqn{\mu_n}$ be a JN-sequence on a  space $X$. Then, there is $N\io$ such that for every $n>N$ the support $\supp\big(\mu_n\big)$ is not a singleton. \noproof
\end{lemma}

\begin{proposition}\label{prop:constant_coeffs}
Let a space $X$ admit a JN-sequence $\seqn{\mu_n}$ such that there exists $M\ge2$, $M\io$, for which we have $\big|\supp\big(\mu_n\big)\big|=M$ for every $n\io$. For each $n\io$ write $\mu_n = \sum_{i=1}^M\alpha^n_i\delta_{x_i^{n}}$. Then, there exist $\alpha_1,\ldots,\alpha_M\in\R$ and an increasing sequence $\seqk{n_k}$  such that the measures $\nu_k = \sum_{i=1}^M\alpha_i\delta_{x_i^{n_k}}$, $k\io$, form a JN-sequence such that $\big\|\nu_k-\mu_{n_k}\big\| \to 0$ as $k\to\infty$.
\end{proposition}

\begin{proof}
Since $C = \{x\in \mathbb{R}^M: \|x\|_1 = 1\}$ is
a compact subspace of $\mathbb{R}^M$, we can find an increasing sequence $\seqk{n_k}$ and a point $\big(\alpha_1,\ldots,\alpha_M\big) \in C$ such that
\[\big(\alpha^{n_k}_1,\dots,\alpha^{n_k}_M\big)\xrightarrow{k\rightarrow\infty} \big(\alpha_1,\ldots,\alpha_M\big)\]
in the norm $\|\cdot\|_1$ of $\R^M$.

Then, for every $k\io$, the measure
\[\nu_k=\sum_{i=1}^M\alpha_i\delta_{x_i^{n_k}}\]
has norm $1$. Notice that $S(\seqk{\nu_k})\sub S(\seqk{\mu_{n_k}})$ and 
\[\big\|\nu_k-\mu_{n_k}\big\|=\sum_{i=1}^M\big|\alpha_i-\alpha_i^{n_k}\big|\xrightarrow{k\rightarrow\infty}0,\]
and appeal to Lemmas \ref{lemma:fsjn_f_bounded} and \ref{lemma:norm_conv_fsjn_seqs} to conclude that $\seqk{\nu_k}$ is a JN-sequence on $X$.
\end{proof}

Note that Proposition \ref{prop:constant_coeffs} does not say that $\alpha_i\neq 0$ for all $i=1,\ldots,M$, but of course we may remove from the definition of each $\nu_k$ all such points $x_i^{n_k}$ for which we have $\alpha_i=0$---by the definition of the support of a measure those point would not belong anyway to $\supp\big(\nu_k\big)$ and so we would have a sequence $\seqn{\nu_n}$ such that $\big|\supp\big(\nu_n\big)\big|<M$ for every $n\io$. Hence, from Proposition \ref{prop:constant_coeffs} we can easily derive the following result.

\begin{corollary}\label{lemma:supports_decrease_sizes}
Let a space $X$ admit a JN-sequence $\seqn{\mu_n}$  such
that there exists $M>2$, $M\io$, for which we have
$\big|\supp\big(\mu_n\big)\big|=M$ for every $n\io$. If there exists
a sequence $\seqn{x_n}$ such that $x_n\in\supp\big(\mu_n\big)$ for every $n\io$ and
$\lim_{n\to\infty}\mu_n\big(\big\{x_n\big\}\big)=0$, then $X$ admits
a JN-sequence $\seqn{\nu_n}$ such that
$\big|\supp\big(\nu_n\big)\big|=M-1$ for every $n\io$.\noproof
\end{corollary}

\begin{remark}\label{remark:2supp_functions}
Note that, by Lemma \ref{lemma:disjoint_supps_size_2}, 
a  space $X$ admits a JN-sequence $\seqn{\mu_n}$ of the form $\mu_n=\frac{1}{2}\big(\delta_{u_n}-\delta_{v_n}\big)$, where $u_n,v_n\in X$ for each $n\io$, if and only if there exist two disjoint sequences $\seqn{x_n}$ and $\seqn{y_n}$ of distinct points in $X$ such that for every $f\in C(X)$ and $\eps>0$ there exists $N\io$ such that for every $n>N$ we have $\big|f\big(x_n\big)-f\big(y_n\big)\big|<\eps$. This observation is
crucial for proving Theorem \ref{theorem:sizes_of_supps}.

If a compact space $K$ is totally disconnected, this fact boils down to the following one: $K$ admits a JN-sequence $\seqn{\mu_n}$ of the form $\mu_n=\frac{1}{2}\big(\delta_{x_n}-\delta_{y_n}\big)$, where $x_n,y_n\in K$, if and only if there exist two disjoint sequences $\seqn{x_n}$ and $\seqn{y_n}$ of distinct points in $K$ such that for every clopen set $U$ there is $N\io$ such that for every $n>N$ either $x_n,y_n\in U$ or $x_n,y_n\in U^c$. 
\end{remark}

In the next lemmas $\kK(\beta X)$ denotes the hyperspace of all non-empty
closed subsets of the \v{C}ech--Stone compactification $\beta X$ of a space $X$, endowed with the Vietoris
topology.

\begin{lemma}\label{lemma:supports_minimal_sizes_limit_points}
Let $X$ be a space for which there exists a JN-sequence $\seqn{\nu_n}$ such that $\sup_{n\io}\big|\supp\big(\nu_n\big)\big|<\infty$. Assume that $M\io$ is the
minimal natural number for which there exists a JN-sequence
$\seqn{\mu_n}$ on $X$ such that $\big|\supp\big(\mu_n\big)\big|=M$ for
every $n\io$. For such a sequence and every $n\io$, put $F_n=\supp\big(\mu_n\big)$. Then,
the set $\fF=\big\{F_n\colon\ n\io\big\}$, treated as a subset of the space $\kK(\beta X)$, has the following two
properties:
\begin{enumerate}
    \item every accumulation point of $\fF$ is a singleton;
    \item $\fF$ is not closed.
\end{enumerate}
\end{lemma}
\begin{proof}
(1) By Corollary \ref{lemma:supports_decrease_sizes} and the minimality of $M$, there exists $\eps>0$ such that for every $n\io$ and $x\in F_n$ we have $\big|\mu_n(\{x\})\big|>\eps$. By Lemma \ref{lemma:supports_never_singletons}, $M>1$. Let $F\in\kK(\beta X)$ be an accumulation point of $\fF$. We claim that $|F|=1$. To see this, let us suppose that $|F|>1$, so there exist distinct $x_0,x_1\in F$. 
Let $U_0$, and $U_1$ be two open subsets of $\beta X$ such that $x_0\in U_0$, $x_1\in U_1$ and $\ol{U_0}\cap\ol{U_1}=\emptyset$. Put:
\[I=\big\{n\io\colon\ F_n\cap U_0\neq\emptyset,\ F_n\cap U_1\neq\emptyset\big\}.\]
Since $F$ is an accumulation point of $\big\{F_n\colon\ n\io\big\}$, $I$ is infinite. Let
$g\in C(\beta X,[0,1])$ be a function such that 
$g\rstr\ol{U_0}\equiv 1$ and $g\rstr\ol{U_1}\equiv 0$. For every
$n\in I$ define the measure $\theta_n$ as follows:
\[\theta_n=(g\rstr X)\der\mu_n\Big/\big\|(g\rstr X)\der\mu_n\big\|.\]
Then, $\seq{\theta_n}{n\in I}$ is a JN-sequence. Indeed, for each $n\in I$ we have $\big\|\theta_n\big\|=1$, and since $F_n\cap U_0=\supp\big(\mu_n\big)\cap (U_0\cap X)\neq\emptyset$, it follows that
\[\big\|(g\rstr X)\der\mu_n\big\|\ge\big\|\big((g\rstr X)\der\mu_n\big)\rstr (U_0\cap X)\big\|=\big\|\mu_n\rstr(U_0\cap X)\big\|>\eps,\]
so if $f\in C(X)$, then for every $n\in I$ we have
$\theta_n(f)=\mu_n\big(f\cdot (g\rstr X)\big)/\big\|(g\rstr X)\der\mu_n\big\|$ and
\[\big|\theta_n(f)\big| = \big|\mu_n\big(f\cdot (g\rstr X)\big)\big|\Big/\big\|(g\rstr X)\der\mu_n\big\| < \big|\mu_n\big(f\cdot (g\rstr X)\big)\big|/\eps.\]
Since $g\rstr X\in C(X)$ and hence
\[\lim_{\substack{n\to\infty\\n\in I}}\mu_n\big(f\cdot(g\rstr X)\big)=0,\]
it follows that
\[\lim_{\substack{n\to\infty\\n\in I}}\theta_n(f)=0.\]
This proves that $\seq{\theta_n}{n\in I}$ is weak* null and hence a JN-sequence.
Since $g\rstr\ol{U_1}\equiv 0$ and for each $n\in I$ it holds that $\supp\big(\theta_n\big)\sub\supp\big(\mu_n\big)$, it follows that $\supp\big(\theta_n\big)\subsetneq\supp\big(\mu_n\big)$, so $\big|\supp\big(\theta_n\big)\big|<M$, which is a contradiction with the assumption that $M$ is minimal. This proves that $F$ is a singleton.

\medskip

(2) By (1), each accumulation point of $\fF$ is a singleton, so since, by Lemma \ref{lemma:supports_never_singletons}, none of the elements of $\fF$ is a singleton, $\fF$ cannot be closed in $\kK(\beta X)$.
\end{proof}

\begin{theorem}\label{theorem:sizes_of_supps}
Let $X$ be a space for which there exists a JN-sequence $\seqn{\mu_n}$ such that 
$\sup_{n\io}\big|\supp\big(\mu_n\big)\big|<\infty$. Then, there exists a JN-sequence $\seqn{\nu_n}$ such that $\nu_n=\frac{1}{2}\big(\delta_{x_n}-\delta_{y_n}\big)$ for every $n\io$, where $x_n,y_n\in X$.
\end{theorem}
\begin{proof}
Let $M\io$ be the minimal natural number for which there exists a JN-sequence
$\seqn{\mu_n}$ on $X$ such that $\big|\supp\big(\mu_n\big)\big|=M$ for
every $n\io$. By Lemma \ref{lemma:supports_never_singletons}, $M>1$. We shall show that $M=2$.

By Corollary \ref{lemma:supports_decrease_sizes} and the minimality of $M$,  there is $\eps>0$ such that, for every $n\io$ and $x\in\supp\big(\mu_n\big)$, it holds $\big|\mu_n(\{x\})\big|>\eps$. For every $n\io$, put $F_n=\supp\big(\mu_n\big)$; then, $|F_n|=M$. Let $\fF=\big\{F_n\colon\ n\io\big\}$; by Lemma \ref{lemma:supports_minimal_sizes_limit_points} every accumulation point of $\fF$ in the Vietoris topology of $\kK(\beta X)$ is a singleton.

For every $n\io$ choose $x_n\neq y_n\in F_n$ and define the measure $\nu_n$ as $\nu_n=\frac{1}{2}\big(\delta_{x_n}-\delta_{y_n}\big)$. We claim that the sequence $\seqn{\nu_n}$ is weak* null and hence a JN-sequence. To see this, assume that there exists $f\in C(X)$ and $\eta>0$ such that the set
\[J=\big\{n\io\colon\ {\textstyle\frac{1}{2}}\big|f\big(x_n\big)-f\big(y_n\big)\big|>\eta\big\}\]
is infinite. Observe that, without loss of generality, we can assume that the function $f$ is bounded. Indeed, by Lemma \ref{lemma:fsjn_f_bounded} $f$ is bounded on $S = S\big(\seqn{\mu_n}\big)$, hence we can replace $f$ by a  bounded function $g\in C(X)$ defined by
\[g = \max\Big(\min\big(f,\sup_{x\in S}f(x)\big), \inf_{y\in S}f(y)\Big),\]
which agrees with $f$ on $S$, in particular $g\big(x_n\big) = f\big(x_n\big)$ and $g\big(y_n\big) = f\big(y_n\big)$ for all $n\io$. 

Let $\beta f\colon \beta X \to \mathbb{R}$ be a continuous extension of $f$, and
let $z\in \beta X$ be such that $\{z\}$ is an accumulation point of the set $\big\{F_n\colon\ n\in J\big\}$ in $\kK(\beta X)$. Let $U$ be a neighborhood of $z$ in $\beta X$ such that for every $x,y\in U$ we have $\big|\beta f(x) - \beta f(y)\big|<2\eta$. Since $\{z\}$ is an accumulation point of $\big\{F_n\colon n\in J\big\}$, there is $n\in J$ such that $F_n\sub U$, and hence $x_n,y_n\in U$, which is a contradiction, as $\big|\beta f\big(x_n\big) - \beta f\big(y_n\big)\big| = \big|f\big(x_n\big) - f\big(y_n\big)\big| > 2\eta$.
\end{proof}

\begin{remark}\label{remark:tot_disc_easier_prof_sizes_2}
Let us note that if $K$ is compact and totally disconnected, then we can prove
Theorem \ref{theorem:sizes_of_supps} without appealing to
Lemma \ref{lemma:supports_minimal_sizes_limit_points}. Indeed, let
$\seqn{\mu_n}$ and $M$ be as in Theorem
\ref{theorem:sizes_of_supps}. By Lemma
\ref{prop:constant_coeffs}, we may assume that there exist non-zero
$\alpha_1,\ldots,\alpha_M\in[-1,1]$ such that for every $n\io$ the
measure $\mu_n$ is of the form
$\mu_n=\sum_{i=1}^M\alpha_i\delta_{x_i^n}$ for some
$x_1^n,\ldots,x_M^n\in K$. Note that for every clopen set $U\sub K$
the sequences $\seqn{\mu_n\rstr U}$ and $\seqn{\mu_n\rstr U^c}$ are
weak* null, so it follows that for sufficiently large $n\io$
either $x_1^n,\ldots,x_M^n\in U$, or $x_1^n,\ldots,x_M^n\in
U^c$---otherwise, we would get a contradiction with the minimality
of $M$. Now, the formula
$\nu_n=\frac{1}{2}\big(\delta_{x_1^n}-\delta_{x_2^n}\big)$ defines a
JN-sequence on $K$,  with the property that
$\big|\supp\big(\nu_n\big)\big|=2$ for every $n\io$. Since $M$ is
minimal, it follows that $M=2$.
\end{remark}


\medskip

\begin{proof}[Proof of Theorem {\ref{theorem:main2}}]
The first part of the thesis follows immediately from Theorem \ref{theorem:sizes_of_supps}. The second part is a consequence of Theorem \ref{theorem:sizes_of_supps} and Remark \ref{remark:2supp_functions}.
\end{proof}

\medskip

Corollary \ref{cor:main2} is an immediate consequence of Theorem \ref{theorem:sizes_of_supps} (or Theorem \ref{theorem:main2}), too.


\begin{thebibliography}{10}

\bibitem{AK06} F.\ Albiac, N.\ Kalton, 
{\em Topics in Banach Space Theory}, 
Graduate Texts in Mathematics 233, Springer, New York, 2006.

\bibitem{BKS19} T. Banakh, J. K\k{a}kol, W. \'Sliwa, {\em Josefson-Nissenzweig property for $C_{p}$-spaces}, RACSAM 113 (2019), 3015--3030.

\bibitem{BKScpxe} C. Bargetz, J. K\k{a}kol, D. Sobota, {\em On complemented copies of the space $c_0$ in spaces $C_p(X,E)$}, to appear in Math. Nachr.

\bibitem{Bat92} E.M. Bator, {\em Unconditionally converging and compact operators on $c_0$}, Rocky Mountain J. Math. 22 (1992), no 2, 417--422.

\bibitem{Beh95} E. Behrends, {\em New proofs of Rosenthal's $\ell^1$-theorem and the Josefson--Nissenzweig theorem}, Bull. Pol. Acad. Sci., Math. 43 (1995), no. 4, 283--295.

\bibitem{Ber71} Y.F. Bereznitski\u{\i}, {\em Nonhomeomorphy between two bicompacta}, Vestnik Moskov. Univ. Ser. I Mat. Mekh. 26 (1971), no. 6, 8--10.

\bibitem{Bon91} J. Bonet, {\em A question of Valdivia on quasinormable Fr\'echet spaces}, Canad. Math. Bull. 34 (1991), no. 3, 301--304.

\bibitem{BLV} J. Bonet, M. Lindstr\"om, M. Valdivia, {\em Two theorems of Josefson--Nissenzweig type for Fr\'echet spaces}, Proc. Amer. Math. Soc. 117 (1993), no. 2, 363--364.

\bibitem{BF93} J.M. Borwein, M. Fabian, {\em On convex functions having points of Gateaux differentiability which are not points of Fr\'echet differentiability}, Canad. J. Math. 45 (1993), no. 6, 1121--1134.

\bibitem{BD84} J. Bourgain, J. Diestel, {\em Limited operators and strict cosingularity}, Math. Nachr. 119 (1984), 55--58.

\bibitem{Cem84} P. Cembranos, {\em $C(K, E)$ contains a complemented copy of $c_0$}, Proc. Amer. Math. Soc. 91 (1984), 556--558.

\bibitem{Die84} J. Diestel, {\em Sequences and Series in Banach Spaces}, Graduate Texts in Mathematics 92, Springer--Verlag, 1984.

\bibitem{En}
R.\ Engelking, {\em General Topology}, Heldermann Verlag, Berlin, 1989.

\bibitem{Fre84} F.J. Freniche, {\em Barrelledness of the space of vector valued and simple functions}, Math. Ann. 267 (1984), 479--486.

\bibitem{HJ77} J. Hagler, W.B. Johnson, {\em On Banach spaces whose dual balls are not weak* sequentially compact}, Israel J. Math. 28 (1977), 325--330.

\bibitem{Jos75} B. Josefson, {\em Weak sequential convergence in the dual of a Banach space does not imply norm convergence}, Ark. Mat. 13 (1975), 79--89.

\bibitem{KMSZ} J. K\k{a}kol, W. Marciszewski, D. Sobota, L. Zdomskyy, {\em On complemented copies of the space $c_0$ in spaces $C_p(X\times Y)$}, Israel J. Math. 250 (2022), 139--177.

\bibitem{KSZprod} J. K\k{a}kol, D. Sobota, L. Zdomskyy, {\em On complementability of $c_0$ in spaces $C(K\times L)$}, to appear in Proc. Amer. Math. Soc.

\bibitem{KSZgroth} J. K\k{a}kol, D. Sobota, L. Zdomskyy, {\em Grothendieck $C(K)$-spaces and the Josefson--Nissenzweig theorem}, to appear in Fund. Math.

\bibitem{Khu78} S.S. Khurana, {\em Grothendieck spaces}, Illinois J. Math. 22 (1978), no. 1, 79--80.

\bibitem{KS12} P. Koszmider, S. Shelah, {\em Independent families in Boolean algebras with some separation properties}, Algebra Universalis 69 (2013), 305--312.

\bibitem{Lev77} R. Levy, {\em Countable spaces without points of first countability}, Pacific J. Math. 70 (1977), no. 2, 391--399.

\bibitem{LS93} M. Lindstr\"om, T. Schlumprecht, {\em A Josefson--Nissenzweig theorem for Fr\'echet spaces}, Bull. London Math. Soc. 25 (1993), 55--58.

\bibitem{MSnf} W. Marciszewski, D. Sobota, {\em The Josefson--Nissenzweig theorem and filters on $\omega$}, to appear in Arch. Math. Logic.

\bibitem{Nis75} A. Nissenzweig, {\em w* sequential convergence}, Israel J. Math. 22 (1975), 266--272.

\bibitem{Sch82} W. Schachermayer, {\em On some classical measure-theoretic theorems for non-sigma-complete Boolean algebras}, Diss. Math. (Rozpr. Mat.) 214 (1982), pp. 34.

\end{thebibliography}
\end{document}